\documentclass[11pt, oneside]{amsart}
\usepackage{amssymb,bm}
\usepackage{amsmath,mathtools} 
\usepackage{mathrsfs,graphicx,enumitem}

\usepackage[colorlinks=true]{hyperref}
\hypersetup{urlcolor=blue, citecolor=red}

\renewcommand{\r}[1]{\eqref{#1}}
\newcommand{\PDO}{$\Psi$DO}
\DeclareMathOperator{\WF}{WF}

\newcommand{\be}[1]{\begin{equation}\label{#1}}
\newcommand{\ee}{\end{equation}}

\newtheorem{Theorem}{Theorem}[section]

\newtheorem{Proposition}[Theorem]{Proposition}

\theoremstyle{definition}
\newtheorem{example}{Example}

\numberwithin{equation}{section}
\renewcommand{\d}{\mathrm{d}}

\newcommand{\bo}{\partial M}
\DeclareMathOperator{\Ker}{Ker}
\def\C {\mathbf C}
\def\R {\mathbf R}

\newcommand{\supp}{\operatorname{supp}}

\renewcommand{\>}{\rangle}

\newcommand{\p}{\partial}
\newcommand{\Vol}{\operatorname{Vol}}

\date{\today}

\begin{document}
\title[Stability estimate for the geodesic ray transform]{Sharp stability estimate for the geodesic ray transform}

\author[Yernat M. Assylbekov]{Yernat M. Assylbekov}
\address{Department of Mathematics, Northeastern University, Boston, MA 02115, USA}
\email{y\_assylbekov@yahoo.com}
\thanks{The work of the first author was partially supported by AMS-Simons Travel Grant}

\author[Plamen Stefanov]{Plamen Stefanov}
\address{Department of Mathematics, Purdue University, West Lafayette, IN 47907, USA}
\email{stefanov@math.purdue.edu}
\thanks{Second author partly supported by  NSF  Grant DMS-1600327}

\begin{abstract} 
We prove a sharp $L^2\to H^{1/2}$ stability estimate for the geodesic X-ray transform of tensor fields of order $0$, $1$ and $2$ on a simple Riemannian manifold with a suitable chosen $H^{1/2}$ norm. We show that such an estimate holds for a family of such $H^{1/2}$  norms, not topologically equivalent, but equivalent on the range of the transform. The reason for this is that the geodesic X-ray transform has a microlocally structured range. 
\end{abstract}

\maketitle

\section{Introduction}

Let $(M,g)$ be a smooth compact $n$-dimensional Riemannian manifold with boundary $\bo$. We assume that $(M,g)$ is \emph{simple}, meaning that $\p M$ is strictly convex and that any two points on $\p M$ are joined by a unique minimizing geodesic.
The weighted  \emph{geodesic ray transform} $I_{m,\kappa}f$ of a smooth covariant symmetric $m$-tensor field $f$ on $M$ is given by
\be{0.1}
I_{m,\kappa} f(\gamma):=\int  \kappa(\gamma (t),\dot\gamma (t)) f_{i_1\dots i_m}(\gamma (t))\dot\gamma^{i_1}(t)\cdots\dot\gamma^{i_m}(t)\,\d t
\ee
where $\kappa$ is a smooth weight, 
 $\gamma$ runs over the set $\Gamma$ of all unit speed geodesics connecting boundary points, and the integrand, written in local coordinates, is invariantly defined. 

When $\kappa=1$, we drop the subscript $\kappa$ and simply write $I_m$. It is well known and can be checked easily that for every $\phi$ regular enough with $\phi=0$ on $\bo$, we have $\d\phi\in \Ker I_1$. Similarly, for every regular enough covector field $v$ of order $m-1$ vanishing at $\bo$, we have $\d^sv\in \Ker I_m$, where $\d^s$ is the symmetrized  covariant differential. Those differentials are called potential fields.  Many works have studied injectivity  of those transforms up to potential fields and stability estimates.

In the present paper, the bundle of symmetric covariant $m$-tensors on $M$ will be denoted by $S^m_M$. If $F$ is a notation for a function space ($H^s$, $C^\infty$, $L^p$, etc.), then we denote by $F(M;S^m_M)$ the corresponding space of sections of $S^m_M$. Note that $S^0_M=\C$ and in this case we simply write $F(M)$ instead of $F(M;S^0_M)$.

The goal of this paper is to prove sharp $L^2(M;S^m_M)\to H^{1/2} $ 
stability estimates for those transforms when $m=0,1,2$ with an appropriate choice of an $H^{1/2}$ space on $\Gamma$. The Sobolev exponents $0$ and $1/2$  are natural in view of the properties of $I_m$ as a Fourier Integral Operator in the interior $M^\textrm{int}$ of $M$.  The complications happen near the boundary. 
Before stating the main results, we will review the known estimates first. 

If $g=e$ is Euclidean, a natural parameterization of the lines in $\R^n$ is as follows:
\be{Sigma}
\ell_{z,\theta}=\{x+t\theta, t\in\R\}, \quad (z,\theta)\in \Sigma:=\{(z,\theta)\in \R^n\times S^{n-1}|\; z\cdot\theta=0\}.
\ee
 One defines the Sobolev spaces $\bar H^s(\Sigma)$ by using derivatives w.r.t.\ $z$ only, see also \r{Hs}. Then, with $I_0^e$ being the Euclidean X-ray transform on functions,
\be{1.1}
\|f\|_{H^s(\R^n)}/C\le \|I_0^ef\|_{\bar H^{s+1/2}(\Sigma)}\le C\|f\|_{H^s(\R^n)},
\ee
for every $f\in C_0^\infty(\Omega)$ with $\Omega\subset \R^n$ a smooth bounded domain, see \cite[Theorem~II.5.1]{Natterer-book} with $C=C(s,n,\Omega)$ (the constant on the right depends on $n$ only). This implies the estimate for every $f\in H^s_{\bar\Omega}$, see the discussion of the Sobolev spaces in Section~\ref{sec_Hs}. It is straightforward to see that the estimate still holds if we define $H^s(\Sigma)$ using all derivatives, including the $\theta$ ones. 

Estimate \r{1.1} was recently proved by Boman and Sharafutdinov \cite{Boman-S} for symmetric tensor fields of every order $m$  for $s=0$ and $f$ replaced by the solenoidal part $f^s$ of $f$, which is the  projection of $f$ to the orthogonal complement of its kernel in $L^2$:
\be{1.2}
\|f^s\|_{L^2(\Omega;S^m_\Omega)}/C\le \|I_m^ef\|_{H^{1/2}(\bar\Sigma)}\le C\|f^s\|_{L^2(\Omega;S^m_\Omega)},
\ee
where $I_m^e$ is the Euclidean ray transform of tensor fields of order $m$ supported in $\bar\Omega$.

In the Riemannian case, injectivity of $I_m$ up to potential fields (called s-injectivity) has been studied extensively, see, e.g., \cite{Sean-X, Monard14, MonardNP-2018,MonardSU14, PestovU04, Sh-UW, Sh-finite,S-AIP, SU-caustics,  SU-Duke, SU-JAMS, SUV-tensors,SUV_anisotropic, UV:local}.  The first proofs of injectivity/s-injectivity of $I_0$ and $I_1$  for simple metrics  in \cite{Muhometov_77, BGerver, AnikonovR} provide a stability estimate with a loss of an $1/2$ derivative. The ray transform there is parameterized by endpoins of geodesics. Another estimate with  a loss of an $1/2$ derivative, conditional when $m=2$, follows from Sharafutdinov's estimate in \cite{Sh-book} for $I_m$, see \r{1.4} below. Stability estimates in terms of the normal operator $N_m = I_m^*I_m$ are established in \cite{SU-Duke}:
\be{1.3}
\|f^s\|_{L^2(M;S^m_M)}/C\le \|N_mf\|_{H^{1}(M_1;S^m_{M_1})}\le C\|f^s\|_{L^2(M;S^m_M)},\quad \forall f\in L^2(M;S^m_M),\quad m=0,1,
\ee
where $M_1\Supset M$ is some extension of $M$ with $g$ extended to $M_1$ as a simple metric. When $m=0$, $f^s=f$ above. In \cite{FSU}, this estimate was extended to the weighted transform $I_{0,\kappa}$, with $\kappa$ never vanishing, under the assumption that the latter is injective, and even to more general geodesic-like  families of curves without conjugate points.  An analogous estimate for the weighted version of $I_1$, assuming injectivity, is proved in \cite{S-Sean-Doppler}. 
Those estimates are based on the fact that $N_m$ is a \PDO\ of order $-1$ elliptic on solenoidal tensor fields (or just elliptic for $m=0$) and injective. The need to have $M_1$ there comes from the fact that the standard \PDO\ calculus is not suited for studying operators on domains with boundary. On the other hand, \PDO s satisfying  the transmission condition can be used for such problems.  In \cite{MonardNP-2018},  it is showed that $N_0$ does not satisfy the transmission condition but satisfies a certain modified version of it. Then one can replace $M_1$ by $M$ in \r{1.3} for $m=0$ at the expense of replacing $H^1$ by a certain H\"ormander type of space. 
A sharp stability estimate for $I_{0,\kappa}: H^{-1/2}(M)\to L^2_\mu(\p_-SM)$ on the orthogonal complement on the kernel is established in \cite{Sean-X}; see next section for the Sobolev norms we use.

The case $m\ge2$ is harder and the $m=2$ one contains all the difficulties already. S-injectivity is known under an a priori upper bound of the curvature \cite{Sh-UW} and also for an open dense set of simple metrics, including real analytic ones \cite{SU-JAMS} (and for a class of non-simple metrics, see \cite{SU-lens}). It was shown in \cite{paternain2013tensor} that $I_m$ is s-injective on arbitrary simple surfaces for all $m\ge 2$. Under a curvature condition, 
Sharafutdinov \cite{Sh-UW} proved the stability estimate 
\be{1.4}
\|f\|_{L^2(M;S^m_M)}\le C\left( \|I_mf\|_{H^1(\p_-SM)} + m(m-1)\|I_mf\|_{L^2(\p_-SM)} \big\|j_\nu f|_{\partial M}\big\|_{L^2(\partial M;S^{m-1}_M)}\right),
\ee
$\forall f\in H^1(M;S^m_M)$, $m=0,1,2$, 
where  $j_\nu f$ equals $f$ ``shortened'' by the unit normal $\nu$ and the spaces 
above are introduced in the next section. This estimate is of conditional type when $m=2$  since $f$ appears on the r.h.s.\ as well; and not sharp since one would expect $I_mf$ to be in some form of an  $H^{1/2}$ norm, as in \r{1.1}. In terms of the normal operator, a non-sharp stability estimate for $I_2$ was established in \cite{SU-JAMS}; and in \cite{S-AIP}, the second author proved the sharp stability estimate \r{1.3} for $m=2$:
\begin{equation}\label{1.5}
\|f^s\|_{L^2(M;S^2_M)}/C\le \|N_2f\|_{H^1(M_1;S^2_{M_1})}\le C \|f^s\|_{L^2(M;S^2_M)}, \quad\forall f\in L^2(S^2_M).
\end{equation}
  The new ingredient in \cite{S-AIP} was to use the Korn inequality estimating $\|v\|_{H^1(M;S^1_M)}$ in terms of $\| v\|_{L^2(M;S^1_M)}+\|\d^s v\|_{L^2(M;S^2_M)}$.

The main result of this paper is a sharp estimate of  the kind \r{1.1}, \r{1.2} (but for  $g$ non- necessarily Euclidean)  for simple metrics and $m=0,1,2$. Our starting point are the estimates \r{1.3} for $m=0,1$ and \r{1.5} for $m=2$. As evident from \r{1.1}, and the remark after it,  there is some freedom in defining the Sobolev space for $I_mf$ since its range is microlocally structured, see also Example~\ref{ex6}. We show that if one defines a Sobolev space using, as derivatives, at least $n-1$ non-trivial Jacobi fields covering $M$ and pointwise linearly independent in $M$, then an analog of \r{1.1} and \r{1.2} holds. 

We define first the space $H^{1/2}_\Gamma$ in the following way. Let $M_1\Supset M$ be  close enough to $M$ so that $g$  extends as a simple metric to $M_1$. We parameterize the maximal unit geodesics $\Gamma_1$ in $M_1$ by initial points on $\bo_1$ and incoming unit directions, i.e., by $\p_-SM_1$, see \r{SM} below.  This defines a smooth structure in the interior of $\Gamma_1$ and natural Sobolev spaces in that interior (see also next section). We define $H^{1/2}_\Gamma$ as the subspace consisting  of the functions supported in $\Gamma$; the latter identified with their initial points and directions on $\p_-SM_1$ (but the geodesics are restricted to $M$).   
While the dot product depends on the extension, the topology does not. 

Clearly, if we try a similar parameterization by $\p_-SM$, we do not get a diffeomorphic relation at the boundary of $\Gamma$ consisting of geodesics tangent to $\bo$ (and having only one common point with $\bo$). 
One can still give an intrinsic definition of $H^{1/2}_\Gamma$   without extending $(M,g)$. 
We parameterize the maximal unit geodesics in $M$ in some neighborhood of the boundary by a point $z$ on each one maximizing the distance to $\bo$ and a unit direction $\theta$ at that point, see also Figure~\ref{pic_par} and section~\ref{sec_par}. One can view this as taking the strictly convex foliation $\text{dist}(\cdot,\bo)=p $, $0\le p\ll1$ first and then taking geodesics tangent to each such hypersurface. For this reason, we call it the foliation parameterization. One can extend it smoothly to geodesic in $M_1\Supset M$, with $g$ extended as a simple metric there, in a natural way. Then we define $H^{1/2}_\Gamma$ as the subspace of $H^{1/2}_0(\Gamma_1)$  consisting of functions supported in $\Gamma$; and this space is topologically equivalent to the previous definition. 
 We refer to section~\ref{sec_par} for more details. The resulting space is independent of the extension $(M_1,g)$.
In Theorem~\ref{thm1} below, we show that one can use prove sharp estimates with the $ H^{1/2}_\Gamma$ norm of $I_{m,\kappa}f$, as a special case. 

The space $ H^{1/2}_\Gamma$ is too large in the Euclidean case, at least, as evident from \r{1.1} and \r{1.2}, where the derivatives used in the definition of $\bar H^{1/2}(\Sigma)$ are  the $z$-ones only. We show that a smaller space can be chosen in the Riemannian case, as well. As mentioned above, we define a space (a family of such, actually) $\bar H^{1/2}_\Gamma$ in section~\ref{sec_HG}, roughly speaking as $H^{1/2}_\Gamma$ but we use  $k\ge n-1$ derivatives on $\p_-SM$ having the properties that the corresponding Jacobi fields are pointwise linearly independent over every point of $M$. If $k=2n-2$, we have  $\bar H^{1/2}_\Gamma=H^{1/2}_\Gamma$ but for $n-1\le k<2n-2$, $\bar H^{1/2}_\Gamma$ is a proper subspace of $ H^{1/2}_\Gamma$. Therefore, those spaces are not topologically equivalent, at least not when $k$ changes; but they are, on the range of $I_m$. 

Our main results is the following. 
\begin{Theorem}\label{thm1}
Let $(M,g)$ be a simple manifold. Let $\bar H_\Gamma^{1/2}$ be any of the spaces defined in section~\ref{sec_HG}. Then 

(a) If $I_{0,\kappa}$ is injective, then for all $f\in L^2(M)$,
\[
\|f\|_{L^2(M)}/C\le \|I_{0,\kappa}f\|_{\bar H^{1/2}_\Gamma}\le C\|f\|_{L^2(M)}.
\] 

(b) For  all 
$f\in L^2(M;S^1_M)$,
\[
\|f^s\|_{L^2(M;S^1_M)}/C\le \|I_1f\|_{\bar H^{1/2}_\Gamma}\le C\|f^s\|_{L^2(M;S^1_M)}.
\] 

(c) If $I_2$ is s-injective, then for  all 
$f\in L^2(M;S^2_M)$,
\[
\|f^s\|_{L^2(M;S^2_M)}/C\le \|I_2f\|_{\bar H^{1/2}_\Gamma}\le C\|f^s\|_{L^2(M;S^2_M)}.
\] 
\end{Theorem}
Note that if $\kappa$ is constant, or more generally related to an attenuation depending on the position only, then $I_{0,\kappa}$ is injective \cite{SaloU11} on surfaces, and $I_1$ is injective, too \cite{AnikonovR}. The transform $I_{0,\kappa}$ is not injective for all, say non-zero, weights \cite{Boman11,Boman93} when $n=2$ and $g$ is Euclidean but it is injective for $n\ge3$ under some conditions on the metric, as it follows from \cite{UV:local,SUV_anisotropic}, for example. 
Injectivity and stability of $I_{1,\kappa}$ has been studied in \cite{S-Sean-Doppler} and the estimate there implies an estimate of the type above which we do not formulate. 
Conditions for injectivity of $I_2$ can be found in \cite{Sh-UW,Sh-book,SU-AJM,SU-JAMS,SUV-tensors}. 

Theorem~\ref{thm1} generalizes \r{1.1} and \r{1.2} to the Riemannian case, in particular. 
In section~\ref{sec_HG}, we present several specific realizations of the $\bar H_\Gamma^{1/2}$ norms with $k<2n-2$ (strictly), i.e., Sobolev spaces defined with fewer derivatives. Whether one can define such  spaces in the natural $\p_-SM$ or $B(\bo)$ parameterizations (see section~\ref{sec_par}) for which Theorem~\ref{thm1} would remain valid, is an open problem, see also  Example~\ref{ex4}. 

\subsection*{Acknowledgments.} The authors thank Gabriel Paternain and Fran\c{c}ois Monard for the discussion about the results in  \cite{MonardNP-2018} and for their helpful comments.

\section{Preliminaries}\label{sec_prel}

\subsection{Sobolev spaces} \label{sec_Hs}
Consider a simple manifold $(M,g)$. Let $SM:=\{(x,v)\in TM:|v|_{g(x)}=1\}$ be its unit sphere bundle and $\p_\pm SM$ be the set of inward/outward unit vectors on~$\p M$,
\be{SM}
\p_\pm SM:=\{(x,v)\in SM:x\in\p M\text{ and }\pm\<v,\nu(x)\>_{g(x)}< 0\},
\ee
where $\nu$ is the inward unit normal to $\p M$. By $\d\Sigma^{2n-1}$ we denote the Liouville volume form on $SM$ and by $\d\Sigma^{2n-2}$ its induced volume form on $\p_\pm SM$. Following \cite{Sh-book},  the Sobolev spaces $H^s(SM)$, $H^s_0(SM)$, $H^{-s}(SM)$ and $H^s(\p_\pm SM)$, $H^s_0(\p_\pm SM)$, $H^{-s}(\p_\pm SM)$, for $s\ge 0$, are the ones  w.r.t. the measures $d\Sigma^{2n-1}$ and $d\Sigma^{2n-2}$, respectively, defined in a standard way.

 We recall that for $s\ge0$, there are several ``natural'' ways to define a Sobolev  space when $\Omega\subset \R^n$ is a domain (or a manifold) with a smooth boundary: $H^s(\Omega)$ is the restriction of distributions in $H^s(\R^n)$  to $\Omega$; next, $H_0^s(\Omega)$ is the completion  of $C_0^\infty(\Omega)$ in  $H^s(\Omega)$; and $H^s_{\bar\Omega}$   is the space of all $u\in H^s(\R^n)$ supported in $\bar\Omega$, also equal to the completion of $C_0^\infty(\Omega)$ in $H^s(\R^n)$, also the space of all $f$ which, extended as zero outside $\bar\Omega$, belong to $H^s(\R^n)$. We have $H^s_{\bar\Omega}= H_0^s(\Omega)$  for $s\not=1/2, 3/2,\dots$, and $H^s_{\bar\Omega}\subset H_0^s(\Omega)\subset H^s(\Omega)$ in general;  and $H_0^s(\Omega)=H^s(\Omega)$ for $0\le s\le 1/2$. 
We have $H^s(\Omega)^* = H^{-s}_{\bar\Omega}$ and $(H^{s}_{\bar\Omega})^* = H^{-s}(\Omega)$, for all $s\in\R$. 
 Those definitions extend naturally to manifolds with boundary. 
We refer to \cite{McLean-book} for more details.

In a similar way, we can define the weighted Sobolev spaces on $H^s_\mu(\p_-SM)$, $s\in\R$. More precisely, for $k\ge0$ integer, $H^k_\mu(\p_-SM)$ is the $H^k$-Sobolev space on $\p_-SM$ w.r.t. the measure $d\mu(x,v):=\<v,\nu(x)\>_{g(x)}\,d\Sigma^{2n-2}(x,v)$. For arbitrary $s\ge 0$, $H^s_\mu(\p_-SM)$ is defined via complex interpolation.
Restricted to distributions supported in the compact $\Gamma$ however, the factor $\mu$ is bounded by below (and above) by a positive constant and it can be removed from the definition.

%

\subsection{Geodesics and the scattering relation}
One way to parameterize the geodesics going from $\p M$ into $M$ is by the set $\p_-SM$, see also Section~\ref{sec_par}. More precisely, for $(x,v)\in\p_-SM)$, we write $\gamma_{x,v}(t)$, $0\le t\le \tau(x,v)$, for the unique geodesic with $x=\gamma_{x,v}(0)$ and $v=\dot\gamma_{x,v}(0)$. Here and in what follows, we set $\tau(x,v):=\max\{t:\gamma_{x,v}(s)\in M\text{ for all } 0\le s\le t\}$ for $(x,v)\in SM$, i.e. the first positive time when $\gamma_{x,v}$ exits $M$. If $(M,g)$ is simple, then $\tau$ is smooth up to the boundary $S\bo$ of $\p_-SM$; more precisely, the extension as $\tau(x,-v)$ to $\p_-SM$ (and extended by continuity on the common boundary) is smooth, see \cite[Lemma 4.1.1]{Sh-book}. Note that $\tau$ is not smooth when $x$  is 
not restricted to $\bo$; the normal derivative has a square root type of singularity at $\bo$.

%
%
%


\subsection{The weighted geodesic ray transform and its adjoint} 
Let $\kappa$ be a smooth function on $SM$. Then the weighted geodesic ray transform $I_{m,\kappa} f$ of $f\in C^\infty(M;S^m_M)$ in \r{0.1} can be expressed in local coordinates  as
$$
I_{m,\kappa} f(x,v)=\int_0^{\tau(x,v)} \kappa(\gamma_{x,v}(t),\dot\gamma_{x,v}(t)) f_{i_1\dots i_m}(\gamma_{x,v}(t))\,\dot \gamma_{x,v}^{i_1}(t)\cdots\dot \gamma_{x,v}^{i_m}(t)\,\d t,\qquad (x,v)\in\p_-SM,
$$
see also \r{0.1}. 
Using Santal\'o formula \cite[Lemma~A.8]{dairbekov2007boundary}, one can see that $I_{m,\kappa}$ is bounded from $L^2(M;S^m_M)$ to $L^2_\mu (\p_-SM)$. 
Its properties as a Fourier Integral Operator suggest that those norms are not sharp, see Proposition~\ref{prop::boundedness of I*}.


Consider the adjoint operator $I_{m,\kappa}^*:L^2_\mu(\p_-SM)\to L^2(M;S^m_M)$. Then again by Santal\'o's formula \cite[Lemma~A.8]{dairbekov2007boundary}, 
\begin{align*}
(I_{m,\kappa}f,w)_{L^2_\mu(\p_-SM)}&=\int_{\p_-SM}\int_0^{\tau(x,v)}\kappa(\gamma_{x,v},\dot\gamma_{x,v}) f_{i_1\dots i_m}(\gamma_{x,v})\,\dot \gamma_{x,v}^{i_1}\cdots\dot \gamma_{x,v}^{i_m}\, \bar w(x,v)\,dt\,d\mu(x,v)\\
&=\int_{SM}\kappa(x,v)f_{i_1\dots i_m}(x)v^{i_1}\cdots v^{i_m}\, \bar w_\psi(x,v)\,d\Sigma^{2n-1},
\end{align*}
where $\psi_w$ is the function on $SM$ that is constant along geodesics and $\psi_w|_{\p_-SM}=w$. 
 Hence, we have
$$
I_{m,\kappa}^*w(x)=\int_{S_x M} v^{i_1}\cdots v^{i_m}\, {\bar\kappa(x,v)}\, \psi_w(x,v)\,d\sigma_x(v),
$$
where $d\sigma_x(v)$ is the measure on $S_x M$ such that $d\sigma_x(v) d\Vol_g(x)=d\Sigma^{2n-1}(x,v)$.

\subsection{Parameterizations of the geodesic manifold} \label{sec_par}
There are three main parameterizations of the set $\Gamma$ of the maximal directed unit speed geodesics  on a simple manifold $(M,g)$. We include geodesics generating to a point corresponding to initial directions tangent to $\bo$ to make $\Gamma$ compact; we call that set $\partial\Gamma$. We recall those three parameterizations below, and we include our foliation one for completeness. Note that  the first three are global and their correctness is guaranteed by the simplicity assumption. 

\begin{description}[itemsep=5pt, topsep=5pt]
\item[{$\p_-SM$ parameterization}]  by initial points and incoming directions.  Each $\gamma\in\Gamma$ is parameterized by an initial point $x\in \bo$ and initial unit direction $v$ at $x$, i.e., by $(x,v)\in \p_-SM$. We write $\gamma=\gamma_{x,v}(t)$, $0\le t\le\tau(x,v)$, where the latter is the length of the maximal geodesic issued from $(x,v)$.

\item[{$B(\bo)$ parameterization}] by initial points and tangential projections of incoming directions.  Each $\gamma\in\Gamma$ is parameterized by an initial point $x\in \bo$ and the orthogonal tangential projection $v'$ of its initial unit direction $v$ at $x$, i.e., by $(x,v')\in B(\bo) $, where $B$ stands for the unit ball bundle. We write somewhat incorrectly  $\gamma=\gamma_{x,v'}$. 

\item[{$\bo\times\bo$ parameterization}] by initial and end points. Each $\gamma\in\Gamma$ is parameterized by its endpoints $x$ and $y$ on $\bo$. If we think of $\gamma$ as a directed geodesic, then the direction is from $x$ to $y$. We  use the notation $\gamma = \gamma_{[x,y]}$.
\item[{foliation parameterization}] Near $\p\Gamma$, let $z$ be the point where the maximum of $\text{dist}(\gamma,\bo)$ is attained, and let $\theta\in SM$ be the direction at $z$. 
 We  use the notation $\gamma = \gamma(\cdot,z,\theta)$. Away from $\p\Gamma$, we can use any of the other parameterizations. We give more details below. 
\end{description}

Identifying $\Gamma$ with the corresponding set of parameters, each one of them being a manifold, introduces a natural manifold structure on it.  While those differential structures are different (near $\partial\Gamma$), the first two ones are homeomorphic but not diffeomorphic.  
In the $\p_-SM$ and in the $B(\bo)$  parameterizations, $\Gamma$ is a compact manifold with boundary $\partial \Gamma$. The boundary in the first one can be removed by allowing geodesics to propagate backwards. 
In the $\bo\times\bo$ one, $\Gamma$ is a compact manifold without a boundary; then $\p\Gamma$ is an incorrect notation and it represents the diagonal. If we project the unit sphere bundle to the unit ball one in the standard way $v\mapsto v'$, the resulting map is not a diffeomorphism up to the boundary, i.e.,  at $v$ tangent to $\bo$. The foliation parameterization makes $\Gamma$ a manifold with a boundary $\p\Gamma$ as well but it allows a natural smooth extension of $\Gamma$ to a smooth manifold  of geodesic $\Gamma_1$ on an extended $M_1\Supset M$, as we show below. 

We describe the foliation parameterization in more detail now. 
Fix a point $q\in\partial M$ and assume that $\partial M$ is strictly convex at $q$ w.r.t.\ $g$. Let $\partial M_1$ be as above. We work in boundary normal coordinates near $q$ in which $q=0$ and $x^n$ is the signed distance to $\partial M$, non-negative in $M$. We can always assume that $\partial M_1$ is given locally by $x^n=-\delta$ with some $0<\delta\ll1$. 
Let $\Gamma_1$ be a small neighborhood of the geodesics tangent to $\partial M$ at $q$ extended until they hit $\partial M_1$. Note that this includes geodesic segments which may lie outside of $M$. 
 We will choose a parameterization of $\Gamma_1$ in the following way. First, since any  geodesic $\gamma\in\Gamma_1$ hits $\partial M_1$ transversally at both ends when $\delta\ll1$, we can parameterize $\gamma$ by its initial point $y'\in\partial M_1$ and incoming unit directions $w$ or their projections $w'$  on $T_{y'}\partial M_1$.
Denote this geodesic by $\gamma_{y',w}$.  
 The foliation  parameterization of $\gamma$ is by $(z,\theta)$, where $z=(z',z^n)$ is the point maximizing the signed distance form $\gamma$ to $\partial M$ (regardless of whether $\gamma$ is entirely outside $M$ or hits $\partial M$), and by the direction $\theta$ at $z$ which must be tangent to the hypersurface $x^n=z^n$. In Figure~\ref{pic_par} on the left, we illustrate this on an almost Euclidean looking example (which is more intuitive) and in the right, we do this in boundary normal coordinates. We call the corresponding geodesic $\gamma(\cdot,z,\theta)$.

\begin{figure}[h!]
  \centering
\includegraphics[page=1]{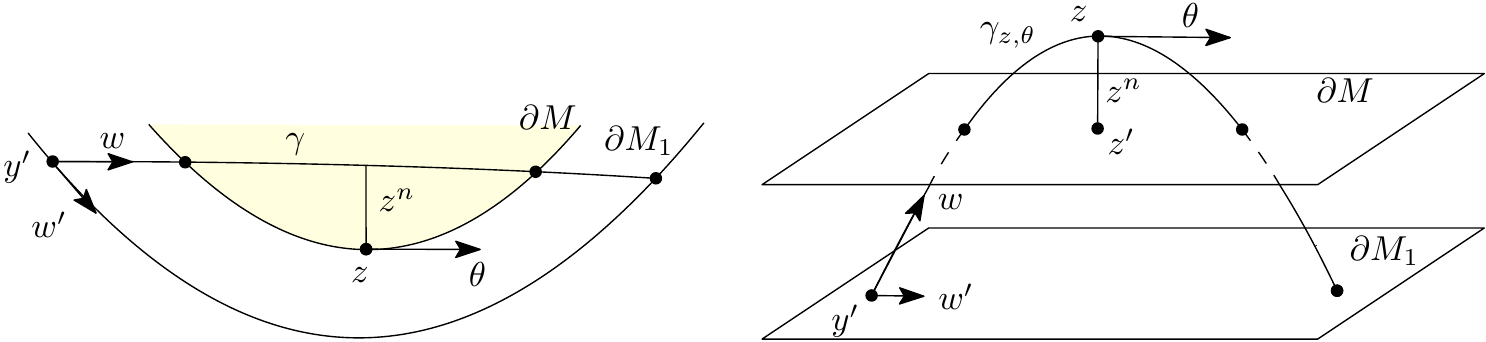}
\caption{The foliation parameterization by $(z,\theta)$}\label{pic_par}
\end{figure}

Another way to describe the foliation parameterization, which explains it name, is to think of the hyperplanes $\Sigma_p:=\{x^n=p\}$, $|p|\ll1$,  as a strictly convex foliation near $q$. Then $\gamma_{z,\theta}$ is the geodesic through $z\in \Sigma_{z^n}$ tangent to it with unit direction $\theta\in S_z\Sigma_{z^n}$. This defines a natural measure on the set of $(z,\theta)$ which we may identify with $\Gamma_1$ given by $\d\Vol_z\, \d\mu_\theta$, where $\d\mu_\theta$ is the natural measure on $S_z\Sigma_p$. Then $(z,\theta)$ belongs locally to the foliation $T\Sigma_p$, $|p|\ll1$ with $p=z^n$, $(z',\theta)\in T\Sigma_p$. 

 Let us compare the $\p_-SM$  parameterization by $(y',w)\in \p_-SM_1$ to the $B(\bo_1)$ one by  $(y',w')$. As we emphasized above, they are related by a diffeomorphism which becomes singular when $w$ is tangent to $\partial M_1$. Such almost tangent geodesics (to $\bo_1$) however do not hit $M$; therefore when parameterizing $If$ with $\supp f\subseteq M$, those two parameterizations are diffeomorphic to each other. 
  
\begin{Proposition}\label{pr_map}
Assume that $\partial M$ is strictly convex at $q$. Then the map $(y',w)\mapsto (z,\theta)$ is a local diffeomorphism. 
\end{Proposition}
\begin{proof}
Let $\tau(y',w)$ be the travel time of the unit speed geodesic issued from $(y',w)\in \p_-SM_1$ to $z$, i.e., $\tau$ maximizes $\gamma^n_{y',w}(\tau)$ locally. Then $\tau$ is a critical point, i.e., $\dot\gamma^n_{y',w}(\tau)=0$. Let $\gamma_{y_0',w_0}$ be a geodesic tangent to $\partial M$ at $q=\gamma_{y_0',w_0}(\tau_0)$ with some $\tau_0$. To solve $\dot\gamma^n_{y',w}(\tau)=0$ for $(y_0',w_0)$ near $(y',w)$, we apply the Implicit Function Theorem. Since  $\ddot\gamma^n_{y_0',w_0}(\tau_0)=-\Gamma_{ij}^n(q)\dot\gamma^i_{y_0',w_0}(\tau_0) \dot\gamma^j_{y_0',w_0}(\tau_0)$ and the latter equals twice the second fundamental form at $q$, we get a unique smooth $\tau(y',w)$ with $\tau(y_0',w_0)=\tau_0$. 

Since $z=\gamma_{y',w}(\tau(  y',w ))$ and $\theta=\gamma'_{y',w}(\tau(  y',w ))$ (the prime stands for the projection onto the first $n-1$ coordinates in boundary normal coordinates), we get that  $(y',w)\mapsto (z,\theta)$ is smooth. 

To verify that the inverse map $(z,\theta) \mapsto (y',w)$ is smooth, it is enough to show that the travel time $t(z,\theta)$ at which $\gamma_{z,\theta}(t)$ reaches $\partial M_1=\{ z^n =-\delta\}$ is a smooth function as well. This follows easily from the fact that geodesics tangent to $\partial M$ hit $\partial M_1$ transversely when $\delta\ll1$. 
\end{proof}

\subsection{The space $H_\Gamma^s$} 
As before, we embed $(M,g)$ in the interior of a simple manifold $(M_1,g)$ (the metric on $M_1$ is an extension of the metric on $M$). We also extend $\kappa$ smoothly to $SM_1$ and keep the same notation for the extension. We denote by $I^{M_1}_{m,\kappa}$ the geodesic ray transform on $M_1$. The set of the oriented geodesics through $M$ will be called $\Gamma$ as before. They are a subset of (the extensions to) all oriented geodesics $\Gamma_1$ in $M_1$. The latter set is parameterized by $\p_-SM_1$; and we identify $\Gamma_1$ with it. In particular, $\Gamma_1$ becomes a manifold with boundary and $\Gamma$ is a compact submanifold contained in its interior. 
On $\p_-SM_1\cong\Gamma_1$ we have two natural measures, as above: one is $\d\Sigma_1^{2n-2}$ and the other one is $\d\mu_1$, the subscript $1$ indicating that they are on $\Gamma_1$. They are equivalent (define equivalent Sobolev spaces) away from $\partial\Gamma_1$. 


Note that the simplicity is not really needed and assuming non-trapping instead of no conjugate points is enough. The strict convexity of $\bo$ is convenient for parameterizing $\Gamma$ on $SM_1$ but that assumption is not needed either, see e.g., \cite{SU-AJM}. 
In the $(z,\theta)$ coordinates, $\Gamma$ is given by $z^n\ge0$. 
For $u$ supported in $\Gamma$, we define the Sobolev space $H_\Gamma^s$ as $H^s_{\bar\Omega}$ above. In particular, when $s$ is a non-negative integer, identifying $\theta$ locally with some parameterization in $\R^{n-1}$, we have locally
\be{Hs}
\|u\|_{H^s_\Gamma}^2 = \int_{(\R^{n-1})^2} \sum_{|\alpha|\le s} |\partial_{z,\theta}^\alpha u|^2  dz\,d\theta.
\ee
This norm is not invariantly defined but under changes of variables, it transforms into equivalent norms.  Note that $u$ above is considered as a function defined on $\Gamma_1$ but supported in $\Gamma$, as in the definition of  $H^s_{\bar\Omega}$ above.

We make this definition global now. Without changing the notation, let $\Gamma_1$ be the manifold of all geodesics with endpoints on $\partial M_1$, and let   $\Gamma$ be those intersecting $M$. We can choose an open cover of $\Gamma$ consisting of neighborhoods of geodesics tangent to $M$ as above, plus an open set $\Gamma_0$ of geodesics passing through interior points only, and having a positive lower bound of the angle they make with $\partial M$. In the latter, we take the classical $H^s$ norm w.r.t.\ the parameterization $(y',w)$, for example. In the former neighborhood, we use the norms $ H^s_\Gamma$ defined above. Then using a partition of unity, we extend the norm $H^s_\Gamma$ to functions defined in $\Gamma_1$ and supported in  $\Gamma$. This defines a Hilbert space which we call $H^s_\Gamma$ again. 

On the other hand, we have the  space $H^s_\Gamma(\p_-SM_1)$ of distributions on $\Gamma$ defined through the parameterization of $\Gamma$ given by $(y',w)\in \p_-SM_1$ in a similar way: we define the $H^s$ norm for functions supported in the interior of $\Gamma_1$ first (the behavior near the boundary of $\Gamma_1$ corresponding to $w$ tangent to $\partial M_1$ does not matter in what follows), and then define $H^s_\Gamma(\p_-SM_1)$ as the subspace of those $u\in H^s(\p_-SM_1)$ which are supported in $\Gamma$. We define $H^{s}({\Gamma^\text{\rm int}})$ similarly, where $\Gamma^\text{\rm int}$ is the interior of $\Gamma$.

Proposition~\ref{pr_map} then implies the following. 

\begin{Proposition}\label{pr_eq}
The Hilbert spaces $H^s_\Gamma(\p_-SM_1)$ and $H^s_\Gamma$ are topologically equivalent. 
\end{Proposition}

\subsection{Mapping properties of $I_{m,\kappa}$ and $I_{m,\kappa}^*$}\label{scn::mapping properties of I and N}
We study the mapping properties of $I_{m,\kappa}$, $I_{m,\kappa}^*$     
restricted a priori to tensor fields/functions supported in fixed compact subsets. This avoids the more delicate question what happens near the boundaries of $M$ and $\Gamma$ but we do not need the latter.

\begin{Proposition}\label{prop::boundedness of I*}
Suppose that $(M,g)$ is simple, $\kappa\in C^\infty(SM)$, and $m\ge 0$. Then, for  $s\ge-1$,

(a)   $I^{M_1}_{m,\kappa} : H^s_M(M, S^m_M)\to H^{s+1/2}_{\Gamma} $ 
is bounded,  

(b)  $(I^{M_1}_{m,\kappa})^*:  
H^{-s-1/2}_{\Gamma}
\to H^{-s}(M_1^\text{\rm int};S^m_M)$ is bounded.
\end{Proposition}
\begin{proof}
Part (a)  is proved in  \cite[Proposition~5.2]{SU-APDE}  for $m=0$ and  $s \ge 0$ but the proof applies to $s\ge -1 $ as well (and it is actually simpler  when $s=-1$). Its tensor version $m\ge1$ is an immediate consequence. 

To prove (b), it is enough to prove that
\be{II}
I^{M_1}_{m,\kappa}  : H^{s}_{M_1}(M_1, S^m_{M_1})\to  H^{s+1/2}({\Gamma^\text{\rm int}}) \quad \text{is bounded},
\ee
then (b) would follow by duality. Here, $I^{M_1}_{m,\kappa} $ is considered as the operator acting on tensor fields supported in $M_1$ restricted to geodesics in $\Gamma^\text{int}$. We can think of $\Gamma^\textrm{int}$ as an open subset of the geodesics $\Gamma_2^\textrm{int}$, with $\Gamma_2$ defined as $\Gamma_1$ but related to an extension $M_2$ of $M_1$. Then by (a),  $I^{M_1}_{m,\kappa}: H^{s}_{M_1}(S^m_{M_1})\to  H^{s+1/2}_{\Gamma_2}$ is bounded, which also proves \r{II} and therefore, (b).

\end{proof}

\section{The spaces $\bar H^{1/2}_\Gamma$} \label{sec_HG}
Let $\gamma_0(t)$, $0\le t\le T$ is a fixed unit speed geodesic on a Riemannian manifold and let $\mathcal{S}$ be a hypersurface intersecting $\gamma_0$ transversely. We are interested in integrals of functions supported in a compact set separated from the endpoints of $\gamma_0$. We parameterize  geodesic (directed)  segments close to $\gamma_0$ (and that parameterization defines the topology) by initial points on $\mathcal{S}$ and initial unit directions. Assume that $\mathcal{S}$ is oriented; then we insist that $t$ increases on the positive side of $\mathcal{S}$. Then we can identify the unit directions with their projection on the unit ball bundle $B\mathcal{S}$. We will apply this construction when $\mathcal{S}$ is a piece of either $\bo$ or $\bo_1$. Note that we exclude geodesics tangent to them in those cases. 

Let $(y^1,\dots,y^{2n-2})$ be local coordinates near a fixed $(z_0,\omega_0)\in B \mathcal{S}$. 
Denote by $\gamma_y(t)$ the geodesics issued from $(z,\omega)$ parameterized by $y$. For some $k\le 2n-2$ fixed, denote $y=(y',y'')$ with $y'=(y^1,\dots,y^k)$, $y''=(y^{k+1},\dots,y^{2n-2})$ with $y'=y$ and $y''$ non-existent if $k=2n-2$. Then we define the $\bar H^{s}_\Gamma$ norm near $(z_0,\omega_0)$ by using $y'$-derivatives only; more precisely for $h$ supported near $(z_0,\omega_0)$, we set 
\be{norm}
\|h\|_{\bar H^{s}}^2 = \int \left( 1+|\xi'|^{2}  \right)^s \left|\mathcal{F}_{y'\to\xi'}h(\xi',y'')\right|^2 \,\d\xi'\,\d y''.
\ee
This is a special case of the spaces introduced in \cite[Definition~10.1.6]{Hormander2} and \cite[Definition~B.1.10]{Hormander3}. 

Given a compact subset $\Gamma\subset  B(\bo_1)$, and a finite  cover of coordinate charts of that kind, 
we use a partition of unity $\chi_j$ to complete that norm to a global one, which we call an $\bar H^{s}_\Gamma$ norm:
\[
\|h\|_{\bar H^{s}_\Gamma}^2 = \sum _j \|\chi_j h\|_{\bar H^s}^2 .
\]
This norm depends on the cover. We are going to require the following non-degeneracy condition in each chart: 
\be{ND}
\text{$\forall y''$, the map $(t,y')\to\gamma_{(y',y'')}(t)$ is a submersion}.
\ee
In other words, the differential of that map has full rank any time when the image is in $M$. Another way to interpret \r{ND} is to say that the Jacobi fields $\partial_{y^j} \gamma_{(y',y'')}(t)$, $j=1,\dots,k$, projected to $\dot \gamma^\perp _{(y',y'')}(t)$, span the latter at every point. 
Clearly,  condition \r{ND} requires $k\ge n-1$. When $k=2n-2$ (no $y''$ variables), \r{ND} hold trivially. 

\begin{example}\label{ex0}
If we use the $\p_-SM_1$ parameterization of $\Gamma$ on $\p M_1$ and take $k=2n-2$, the space $\bar H^{1/2}_\Gamma$ reduces to $H_\Gamma^{1/2}(\p_-SM_1)$. The latter is equivalent to $H^{1/2}_\Gamma$ defined through the foliation parameterization, see Proposition~\ref{pr_eq}.
\end{example}

\begin{example}\label{ex1}
The classical parameterizations of lines in the Euclidean case by $\Sigma$, see \r{Sigma}, is an example of such a coordinate system. In this case, $z$ belongs to the hyperplane $\theta^\perp$ depending on $\theta$ but near a fixed $\theta$, one can always construct a  local diffeomorphism  smoothly depending on $\theta$ allowing us to think of $z$ as a variable on a fixed hyperplane. If that diffeomorphism is a unitary map for each $\theta$ (which can be done), then this would not affect the definition of \r{norm}.  Then we set $z=y'\in \R^{n-1}$ and choose $y''\in\R^{n-1}$ to be a local parameterization of $\theta$. The map in \r{ND} is given by $(t,z)\mapsto z+t\theta$ which is a diffeomorphism. Then the resulting space $\bar H^{1/2}_\Gamma$ is the one appearing in \r{1.3}. Here, and in the examples below, $k=n-1$.
\end{example}

\begin{example}\label{ex2}
Near a point on $\p_-SM_1$ (or, equivalently, on $BM_1$), we choose coordinates $y''\in \R^{n-1}$ to parameterize points on $\bo_1$ and $y'\in \R^{n-1}$ to parameterize incident unit directions. Then the map \r{ND} is 
a submersion when its image is restricted to $M$ by the simplicity of $(M_1,g)$, which can be guaranteed if the extension is close enough to $(M,g)$. Note that we need the initial points of the geodesics to be outside $M$ since \r{ND} is the exponential map in polar coordinates, rather than in the usual ones, and it is not an submersion when $t=0$.
The resulting $\bar H^{1/2}_\Gamma$ space would involve derivatives w.r.t.\ the direction (but not w.r.t.\ the base point) only.  
While the specific definition of the norm depends on the coordinates used, a change would yield an equivalent norm. One can think of those coordinates as fan-beam ones on $\p M_1$ but we use the directions only to define $\bar H^{1/2}_\Gamma$. 
\end{example}

\begin{example}\label{ex3}
With $M_1$ as above, we swap $y'$ and $y''$. More precisely, near a point on $\partial_- SM_1$ (or, equivalently, on $BM_1$), we choose coordinates $y'\in \R^{n-1}$ to parameterize points on $\bo_1$ and $y''\in \R^{n-1}$ to parameterize incident unit directions. The corresponding Sobolev space $\bar H^{1/2}_\Gamma$ will include derivative w.r.t.\ initial points on $\p M_1$ only \textit{in the chosen coordinate system}. A change of variables would include directional derivatives as well. For rays close enough to ones tangential to $\bo$, \r{ND} will hold by a perturbation argument. Then we use a partition of unity do define $\bar H^{1/2}(\Gamma)$ near the boundary of $\Gamma$ (consisting of geodesics tangent to $\p M$ generating to points). In the interior of $\Gamma$, we can swap $y'$ and $y''$ (use derivatives w.r.t.\ the directions) or use all variables, as in   Example~\ref{ex1}. 
\end{example}

 \begin{example}\label{ex4}
One may wonder if one of the parameterizations on $\bo$ (rather than on $\bo_1$) would work as well. If we view the $\p_-SM$ parameterization as an $\p_-SM_1$ projected onto $\p_-SM$, then the natural measure would be $d\mu$, see section~ \ref{sec_Hs}. The derivatives w.r.t.\ initial points on $\bo$ (see also Figure~\ref{pic_par}, with  directions tangent to $\bo$ fixed in some coordinate system (or varying smoothly) would correspond to Jacobi fields which do not satisfy \r{ND} at $t=0$, as it is easy to see. The  directional derivatives at every fixed $x\in M$ generate Jacobi fields vanishing at $t=0$. Therefore, \r{ND} is not satisfied even if we take all possible derivatives. This shows that the space $H^{1/2}_{\mu,\overline{\p_-SM}}$ (here, $\mu$ stands for the measure) does not satisfy \r{ND}. Note that such a space can also be defined as an complex interpolation space between similar spaces with $s=0$ and $s=1$, see \cite[Theorem~B.9]{McLean-book}, where one can use classical definitions of norms through derivatives. The fact that \r{ND} fails in this case does not prove that we cannot use the $H^{1/2}_{\mu,\overline{\p_-SM}}$  norm in our main results yet however. 
\end{example}

\section{Proof of the main theorem}\label{sctn::new parameterization}
\begin{proof}
The starting point are the  stability estimates \r{1.3} for $m=0,1$ and  \r{1.5} for $m=2$, the latter due to the second author \cite[Theorem~1]{S-AIP}, valid for all symmetric 2-tensor field $f\in L^2(M; S^2_M)$. 
First we will estimate $\|N^{M_1} f\|_{H^1(M_1)}$ in the first inequality in \r{1.3}, respectively \eqref{1.5}, by $C\|I f\|_{ H^{1/2}_\Gamma}$, see \r{Hs}, with the corresponding ray transform $I$. We will take $m=2$ below and the proof is the same for $m=0,1$.

By Proposition~\ref{prop::boundedness of I*}(a), applied to the extension of $f\in L^2(M;S^2_M)$  by zero to $M_1\setminus M$, we have $I_2^{M_1}f\in H^{1/2}_\Gamma$ and that map is continuous.  This proves the second inequality in the theorem, part (c), because $I_2f=I_2f^s$. 
We also have  $\supp(I_2^{M_1}f)\subset \Gamma$. 
Applying  Proposition~\ref{prop::boundedness of I*}(b) with $s=-1$ to the middle term  of \eqref{1.5}, we obtain
\begin{equation}  \label{12}
\|f^s\|_{L^2(M;S^2_M)}\le C\|I_2^{M_1} f\|_{H^{1/2}_\Gamma},\quad f\in L^2(M;S^2_M).
\end{equation}
This completes the proof of the first inequality in the theorem with the $H^{1/2}_\Gamma$ norm, i.e., when $k=2n-2$ (in all charts). Then the norm $\|\cdot\|_{\bar H^s_\Gamma}$ is equivalent to  $\|\cdot\|_{H^{s}_\Gamma}$. 

In the remainder of the proof, we consider the more interesting case when $k<2n-2$ and \r{ND} holds.  Then the norm $\|\cdot\|_{\bar H^s_\Gamma}$ is not  equivalent to  $\|\cdot\|_{H^{s}_\Gamma}$ anymore. 
The main idea is that in that case, locally, while $(\partial_{y^1},\dots, \partial_{y^k})$ is not elliptic (in $\R^{2n-2}_{y',y''}$), it is elliptic on the Lagrangian of $I_{\kappa,n}$; more precisely on the image of $T^*M\setminus 0$ under the canonical relation $C$,  where $\WF(I_{\kappa,n}f)$ lies.

We can view $I_{m,\kappa}$ as a sum of several weighted geodesic X-ray transforms of the scalar components of the tensor $f$ (in a  coordinate system near a fixed geodesic). It was shown in \cite{MonardSU14} that each such transform, and therefore $I_{m,\kappa}$ itself, is an FIO with the following canonical relation $C$. Let $(\zeta, \omega)$ be the dual variables of $(z,w)$, and $\xi$ be the dual of $x$. Then $(z,w,\zeta,\omega; x,\xi)\in C$, if and only if there exists $t$ so that
\[
x=\gamma_{z,w}(t), \quad \xi_j\dot\gamma^j_{z,w}(t)=0, \quad \zeta_\alpha= \xi_j \frac{\partial \gamma^j_{z,w}(t)}{\partial z^\alpha}, \quad\omega_\alpha =  \xi_j \frac{\partial \gamma^j_{z,w}(t)}{\partial w^\alpha}. 
\]
In particular, this shows that the dual variables $(\zeta, \omega)$ along each geodesic are Jacobi fields projected to its conormal bundle. Passing to the $y$ variables, the last two equations become 
\[
\eta_\ell  =  \xi_j \frac{\partial \gamma^j_{y}(t)}{\partial y^\ell}, \quad \ell=1,\dots,2n-2,
\]
where $\eta$ is the dual to $y$. 
We can describe $C$ in the following way. For a fixed $(x,\xi)$, choose any $\gamma^j_{z,w}$ through $x$ normal to $\xi$ (that set is diffeomorphic to $S^{n-2}$); then $C(x,\xi)$ is the union of all  $(y,\eta)$  so that $y$ parameterizes some of those geodesics and $\eta_\ell$ is its Jacobi field corresponding to $\partial_{y^\ell}$ at $x$ projected to $\xi$, see \cite{MonardSU14}. By our assumption, for every such fixed geodesic, at least one of those projections corresponding to $\ell=1,\dots,k$,  would not vanish. This means that $\Delta_{y'}$ is elliptic on the image of $T^*M\setminus 0$ under $C$ (which is conically compact, and therefore,  $\Delta_{y'}$ it is elliptic  in a neighborhood of it), where $\WF(I_{m,\kappa}f)$ lies. Therefore, given $s$,  one can build a left parametrix $A$ of order zero  to get
\be{BA1}
A (1-\Delta_{y'})^s \chi_j I_{m,\kappa}f =(1-\Delta_y )^s \chi_j  I_{m,\kappa}f + Rf
\ee
with $R:C_0^\infty(M)\to C_0^\infty(M_1)$ smoothing. Here the fractional powers are defined through the Fourier transform and $A$ is properly supported in some neighborhood of $\supp\chi_j\times\supp\chi_j$. Summing up, we get the estimate
\be{ES1}
\|I_{m,\kappa}f \|_{H^s}\le C \|I_{m,\kappa}f \|_{\bar H_\Gamma^s}+ C_N\|f\|_{H^{-N}(M_1)}
\ee
with $N$ as large as we want. On the other hand, the estimate
\be{ES2}
\|I_{m,\kappa}f \|_{\bar H_\Gamma^s} \le C\|I_{m,\kappa}f \|_{H^s}
\ee
is immediate. 

Since we proved Theorem~\ref{thm1} with the $H^{1/2}_\Gamma$  norm (when $k=2n-2$), by \r{ES1}, \r{ES2}, we can replace that norm by any norm of the $\bar H^{1/2}_\Gamma$ ones at the expense of getting an error term; for (c) in Theorem~\ref{thm1}, for example, we get
\[
\|f^s\|_{L^2(M;S^2_M)}/C\le \|I_2f\|_{\bar H^{1/2}_\Gamma}\le C\|f^s\|_{L^2(M;S^2_M)}+ C_N\|f^s \|_{H^{-N}(M_1)}
\] 
$\forall N$, since $I_2f=I_2f^s$. Since $I_2: \mathcal{S}L^2(M;S_M^2)\mapsto \bar H^{1/2}_\Gamma$, where $\mathcal{S}$ is the projection onto the solenoidal tensors,  is injective, a standard functional analysis argument implies that the last term can be removed at the expense of increasing $C$. 
\end{proof}

The estimate which we prove and even the Euclidean estimate \r{1.1} may look unexpected. The transform $I_{\kappa,m}$ is overdetermined in the sense that it acts from an $n$ dimensional space to an $2n-2$ dimensional one. One could expect that $n$ derivatives in the definition of the Sobolev spaces of the image would be enough but it turns out that $n-1$ suffice, under condition \r{ND}. In the next example, we demonstrate, in a simple situation, that not only the dimension of the Lagrangian projected on the image matters; its structure is the one allowing us to get away with one less variable. 

\begin{example}\label{ex6}
We will demonstrate explicitly how this argument works for the Radon transform $R$ in $\R^2$ with the ``parallel geometry'' paremeterization $x\cdot\omega=p$, $|\omega|=1$, $p\in \R$. We parameterize $\omega$ by its polar angle $\varphi$ and denote by $(\hat p, \hat\varphi)$ the dual variables to $(p,\omega)$. It is well known, and also follows from the analysis above that $R$ is an FIO with a canonical relation which is a local diffeomorphism. A direct computation \cite{S-sampling} shows that under the a priori assumption $\supp f\in B(0,R)$, we have
\be{R10}
 \WF(Rf)\subset \left\{(\varphi,p,\hat\varphi,\hat p); \; |\hat\varphi|\le R|\hat p|\right\}.
\ee
The symbol $\hat p$ (corresponding to $-i \partial_p)$ is not elliptic because it vanishes on the line $\hat p=0$, $\hat\varphi\not=0$. On the other hand, that line is separated from the cone in the r.h.s.\ of \r{R10}. Using a partition of unity on the unit circle $|\hat\varphi|^2+|\hat p|^2=1$ one can always  modify $\hat p$ away from that cone to make it elliptic of order $1$ (for example, by adding a suitable elliptic pure imaginary symbol away from the cone to ensure ellipticity in the transition region as well  where the cutoff is neither $0$ nor $1$). This would result in a smoothing error applied to $f$; and will lead to an elliptic extension of $\hat p$. This shows that defining an $H^{1/2}$ Sobolev space for $Rf$ with the $p$ derivative only should work; and this is a partial case of \r{1.3} written in the $(p,\varphi)$ coordinates. 
\end{example}
 
We recall that the main argument in the proof of the main theorem was that $\Delta_{y'}$ was elliptic on the range of the canonical relation $C$ (away from the zero section). In Example~\ref{ex6}, $y'=p$ and clearly, the dual variable $\hat p$ does not vanish on \r{R10}. In dimensions $n=2$, that range has dimension $4$, which is also the dimension of $T^*M$ and also of $T^*\Gamma$, see also \cite{MonardSU14}. For general dimensions $n\ge2$, this microlocal range has dimension $3n-2$ as it follows from \cite{MonardSU14}; and when $n\ge3$, this is strictly smaller than the dimension $4n-4$ of the phase space $T^*\Gamma$ of all geodesics.


\end{document}